\documentclass[a4paper,draft,reqno,12pt]{amsart}
\usepackage[english]{babel}
\usepackage{amsmath}
\usepackage{amssymb}
\usepackage{amscd}
\usepackage{amsthm}
\usepackage{euscript}
\newtheorem{proposition}{Proposition}
\newtheorem{conjecture}{Conjecture}
\newtheorem{problem}{Problem}
\newtheorem{lemma}{Lemma}
\newtheorem{claim}{Claim}
\newtheorem{theorem}{Theorem}

\newtheorem{corollary}{Corollary}
\theoremstyle{definition}

\theoremstyle{remark}
\newtheorem {remark}{Remark}

\DeclareMathOperator{\Spec}{Spec}

\DeclareMathOperator{\Aut}{Aut}
\DeclareMathOperator{\SAut}{SAut}
\DeclareMathOperator{\LND}{LND}

\DeclareMathOperator{\GL}{GL}

\def\Ker{{\rm Ker}}

\def\PGL{{\rm PGL}}

\def\GG{{\mathbb G}}

\def\CC{{\mathbb C}}
\def\KK{{\mathbb K}}
\def\TT{{\mathbb T}}
\def\ZZ{{\mathbb Z}}

\def\PP{{\mathbb P}}
\def\AA{{\mathbb A}}

\def\OO{\mathcal{O}}

\begin{document}
\date{}
\title[Infinite transitivity and special automorphisms]{Infinite transitivity and special automorphisms}
\author{Ivan Arzhantsev}
\thanks{The research was supported by the grant RSF-DFG 16-41-01013}
\address{National Research University Higher School of Economics, Faculty of Computer Science, Kochnovskiy Proezd 3, Moscow, 125319 Russia}
\email{arjantsev@hse.ru}

\subjclass[2010]{Primary 14J50, 14M17; \ Secondary 13A50, 14L30, 14R20}

\keywords{Quasiaffine variety, automorphism, transitivity, torus action, rigidity}

\maketitle

\begin{abstract}
It is known that if the special automorphism group $\SAut(X)$ of a quasiaffine variety $X$ of dimension at least $2$ acts transitively on $X$, then this action is infinitely transitive. In this paper we address the question whether this is the only possibility for the automorphism group $\Aut(X)$ to act infinitely transitively on $X$. We show that this is the case provided $X$ admits a nontrivial $\GG_a$- or $\GG_m$-action. Moreover, 2-transitivity of the automorphism group implies infinite transitivity.
\end{abstract}

\section{Introduction}

Consider a set $X$, a group $G$ and a positive integer $m$. An action $G\times X\to X$ is said to be $m$-transitive if it is transitive on ordered $m$-tuples of pairwise distinct points in $X$, and is infinitely transitive if it is $m$-transitive for all positive integers $m$.

It is easy to see that the symmetric group $S_n$ acts $n$-transitively on a set of order $n$, while the action of the alternating group $A_n$ is $(n-2)$-transitive. A generalization of a classical result of Jordan~\cite{Jo} based on the classification of finite simple groups claims that there
are no other $m$-transitive finite permutation groups with $m>5$.

Clearly, the group $S(X)$ of all permutations of an infinite set $X$ acts infinitely transitively on $X$. The first explicit example of an infinitely transitive and faithful action of the free group $F_n$ with the number of generators $n\ge 2$ was constructed in~\cite{McD}; see \cite{FMS,HO} and references therein for recent results in this direction.

Infinite transitivity on real algebraic varieties was studied in \cite{HM,HM1,BM,KM}. For multiple transitive actions of real Lie groups on real manifolds, see~\cite{Bo,Kram}.

A classification of multiple transitive actions of algebraic groups on algebraic varieties over an algebraically closed field is obtained in~\cite{Kn}. It is shown there that the only 3-transitive action is the action of $\PGL(2)$ on the projective line $\PP^1$. Moreover, for reductive groups the only 2-transitive action is the action of $\PGL(m+1)$ on $\PP^m$.

In this paper we consider highly transitive actions in the category of algebraic varieties over an algebraically closed field $\KK$ of characteristic zero. By analogy with the full permutation group $S(X)$ it is natural to ask about transitivity properties for the full automorphism group $\Aut(X)$ of an algebraic variety $X$. The phenomenon of infinite transitivity for $\Aut(X)$ in affine and quasiaffine settings was studied in many works, see~\cite{Re,KZ,AKZ,AFKKZ,AFKKZ1,FKZ,APS}. The key role here plays the special automorphism group $\SAut(X)$.

More precisely, let $\GG_a$ (resp. $\GG_m$) be the additive (resp. multiplicative) group of the ground field $\KK$. We let $\SAut(X)$ denote the subgroup of $\Aut(X)$ generated by all algebraic one-parameter unipotent subgroups of $\Aut(X)$, that is, subgroups in $\Aut(X)$ coming from all regular actions $\GG_a\times X\to X$.

Let $X$ be an irreducible affine variety of dimension at least 2 and assume that the group $\SAut(X)$ acts transitively on the smooth locus $X_{\text{reg}}$. Then \cite[Theorem~0.1]{AFKKZ} claims that the action is infinitely transitive. This result can be extended to quasiaffine varieties; see \cite[Theorem~2]{APS} and \cite[Theorem~1.11]{FKZ}.

We address the question whether transitivity of $\SAut(X)$ is the only possibility for the automorphism group $\Aut(X)$ of an irreducible quasiaffine variety $X$ to act infinitely transitively on $X$. We show that 2-transitivity of the group $\Aut(X)$ implies transitivity
of the group $\SAut(X)$ provided $X$ admits a nontrivial $\GG_a$- or $\GG_m$-action; (Theorem~\ref{tmain} and Corollary~\ref{ctrans}). We conjecture that the assumption on existence of a nontrivial $\GG_a$- or $\GG_m$-action on $X$ is not essential and 2-transitivity of $\Aut(X)$ always implies transitivity of $\SAut(X)$ and thus infinite transitivity of $\Aut(X)$ (Conjecture~\ref{conj}).

The quasiaffine case differs from the affine one at least by two properties: the algebra of regular functions $\KK[X]$ need not be finitely generated and not every locally nilpotent derivation on $\KK[X]$ gives rise to a $\GG_a$-action on $X$. These circumstances require new ideas when transferring the proofs obtained in the affine case. Our interest in the quasiaffine case, especially when the algebra $\KK[X]$ is not finitely generated, is motivated by several reasons.
Homogeneous quasiaffine varieties appear naturally as homogeneous spaces $X=G/H$ of an affine algebraic group $G$. By Grosshans' Theorem, the question whether the algebra $\KK[G/H]$ is finitely generated is crucial for the Hilbert's fourteenth problem, see~\cite{Gr} and \cite[Section~3.7]{PV}. The group $\Aut(X)$ acts infinitely transitively on $X$ provided the group $G$ is semisimple \cite[Proposition~5.4]{AFKKZ}. On the other hand, quasiaffine varieties, including the ones with not finitely generated algebra of regular functions, appear as universal torsors $\widehat{X}\to X$ over smooth rational varieties $X$ in the framework of the Cox ring theory, see e.g. \cite[Propositions~1.6.1.6, 4.3.4.5]{ADHL}. By \cite[Theorem~3]{APS}, for a wide class of varieties $\widehat{X}$ arising in this construction, the special automorphism group $\SAut(\widehat{X})$ acts infinitely transitively on $\widehat{X}$.

Let us give a short overview of the content of the paper. In Section~\ref{s1} we recall basic facts on the correspondence between $\GG_a$-actions on an affine variety $X$ and locally nilpotent derivations of the algebra $\KK[X]$. Proposition~\ref{lndga} extends this correspondence to the case when $X$ is quasiaffine.

In Section~\ref{s2} we generalize the result of \cite{AG} on the automorphism group of a rigid affine variety to the quasiaffine case. Recall that an irreducible algebraic variety $X$ is called rigid if $X$ admits no nontrivial $\GG_a$-action. Theorem~\ref{trigid} states that the automorphism group of a rigid quasiaffine variety contains a unique maximal torus; the proof is an adaptation of the method of  \cite[Section~3]{FZ1} to our setting.

Also we describe all affine algebraic groups which can be realized as a full automorphism group of a quasiaffine variety (Proposition~\ref{pdref}); the list of such groups turns out to be surprisingly short.

Section~\ref{s3} contains our main results, Theorem~\ref{tmain} and Corollary~\ref{ctrans}. In Corollary~\ref{cunirat} we observe that if an irreducible quasiaffine variety $X$ admits a nontrivial $\GG_a$- or $\GG_m$-action, the group $\Aut(X)$ acts on $X$ with an open orbit $\OO$, and the action of $\Aut(X)$ is 2-transitive on $\OO$, then $X$ is unirational. This result follows also from~\cite[Corollary~3]{Po}.

In the last section we discuss some questions related to Conjecture~\ref{conj}. We pose a problem on transitivity properties for the automorphism group on a quasiaffine variety with few locally finite automorphisms (Problem~\ref{p1}) and ask about classification of homogeneous algebraic varieties (Problem~\ref{p2}).

The author would like to thank Sergey Gaifullin, Alexander Perepechko, Andriy Regeta and Mikhail Zaidenberg for helpful comments and remarks. Also he is grateful to the anonymous referee for valuable suggestions.

\section{Locally nilpotent derivations and $\GG_a$-actions} \label{s1}

In this section we discuss basic facts on locally nilpotent derivations and $\GG_a$-actions on quasiaffine varieties; see~\cite[Section~1.1]{FKZ}, \cite[Section~2]{APS}, and \cite{DL} for related results.

Let $A$ be a $\KK$-domain and $\partial\colon A\to A$ a derivation, i.e., a linear map satisfying the Liebniz rule $\partial(ab)=\partial(a)b+a\partial(b)$ for all
$a,b\in A$. The derivation $\partial$ is called locally nilpotent if for any $a\in A$ there exists a positive integer $m$ such that $\partial^m(a)=0$. Let us denote the set of all locally nilpotent derivations of $A$ by $\LND(A)$. Clearly, if $\partial\in\LND(A)$ and $f\in\Ker(\partial)$, then $f\partial\in\LND(A)$.

Every locally nilpotent derivation defines
a one-parameter subgroup $\{\exp(s\partial),\, s\in\KK\}$ of automorphisms of the algebra $A$.
This subgroup gives rise to an algebraic action of the group $\GG_a$ on the algebra $A$.
The latter means that every element $a\in A$ is contained in a finite dimensional $\GG_a$-invariant subspace $U$ of $A$, and the $\GG_a$-module $U$ is rational. Conversely, the differential of an algebraic $\GG_a$-action on $A$ is a locally nilpotent derivation; see \cite[Section~1.5]{F} for details.

Assume that the domain $A$ is finitely generated and $X=\Spec(A)$ is the corresponding irreducible affine variety. The results mentioned above establish a bijection between locally nilpotent derivations on $A$ and algebraic actions $\GG_a\times X\to X$. Moreover, the algebra of invariants $A^{\GG_a}$ coincides with the kernel of the corresponding locally nilpotent derivation.

If $X$ is an irreducible quasiaffine variety, then again every action $\GG_a\times X\to X$ defines a locally nilpotent derivation of $A:=\KK[X]$. Since regular functions separate points on $X$,
such a derivation determines a $\GG_a$-action uniquely. At the same time, not every locally nilpotent derivation of $A$ corresponds to a $\GG_a$-action on $X$. For example, the derivation
$\frac{\partial}{\partial x_2}$ of the polynomial algebra $\KK[x_1,x_2]$ does not correspond to
a $\GG_a$-action on $X:=\AA^2\setminus\{(0,0)\}$, while the derivation $x_1\frac{\partial}{\partial x_2}$ does.

\smallskip

The following result seems to be known, but for lack of a precise reference we give it with a complete proof.

\begin{proposition} \label{lndga}
Let $X$ be an irreducible quasiaffine variety and $A=\KK[X]$. Then
\begin{enumerate}
\item[(i)]
for every $\partial\in\LND(A)$ there exists a nonzero $f\in\Ker(\partial)$ such that the locally nilpotent derivation $f\partial$ corresponds to a $\GG_a$-action on $X$;
\item[(ii)]
if $\partial\in\LND(A)$ corresponds to a $\GG_a$-action on $X$, then for every $f\in\Ker(\partial)$ the derivation $f\partial$ corresponds to a $\GG_a$-action on $X$.
\end{enumerate}
\end{proposition}

\begin{proof}
We begin with~(i). Fix a derivation $\partial\in\LND(A)$ and the corresponding $\GG_a$-action on~$A$. Consider an open embedding $X\hookrightarrow Z$ into an irreducible affine variety $Z$. Fix a finite dimensional $\GG_a$-invariant subspace $U$ in $A$ containing a set of generators of
$\KK[Z]$. Let $B$ be the subalgebra in $A$ generated by $U$ and $Y$ be the affine variety
$\Spec(B)$. Since $B$ is $\GG_a$-invariant, we have the induced $\GG_a$-action on $Y$. The inclusion $B\subseteq A$ defines an open embedding $X\hookrightarrow Y$.

\begin{claim}
Every divisor $D\subseteq Y$ contained in $Y\setminus X$ is $\GG_a$-invariant.
\end{claim}

\begin{proof}
Assume that the variety $Y$ is normal and take a function $f\in\KK(Y)$ which has a pole along
the divisor $D$. Multiplying $f$ by a suitable function from $B$ we may suppose that $f$ has
no pole outside $D$. Then $f$ is contained in $A$. If the divisor $D$ is not $\GG_a$-invariant,
there is an element $g\in\GG_a$ such that $g\cdot D$ intersects $X$. It shows that the function $g\cdot f$ has a pole on $X$ and thus is not in $A$, a contradiction.

If $Y$ is not normal, we lift the $\GG_a$-action to the normalization of $Y$ and apply the same arguments to integral closures of $A$ and $B$.
\end{proof}

\begin{claim}
There is an open $\GG_a$-invariant subset $W\subseteq Y$ which is contained in $X$.
\end{claim}

\begin{proof}
Let $F$ be the union of irreducible components of $Y\setminus X$ of codimension at least $2$.
Then the closure $\overline{\GG_a\cdot F}$ is a proper closed $\GG_a$-invariant subset whose complement intersected with $X$ is the desired subset $W$.
\end{proof}

Let $Y_0:=Y\setminus W$. This is a closed $\GG_a$-invariant subvariety in $Y$ and its ideal
$I(Y_0)$ in $B$ is a $\GG_a$-invariant subspace. Applying the Lie-Kolchin Theorem, we find a nonzero $\GG_a$-invariant function $f\in I(Y_0)$. Then $f\in\Ker(\partial)$ and the $\GG_a$-action on $Y$ corresponding to the derivation $f\partial$ fixes all points outside $W$. In particular,
this action induces a $\GG_a$-action on $X$. This proves~(i).

Now we come to~(ii). Consider the action $\GG_a\times X\to X$ corresponding to $\partial$.
By~\cite[Theorem~1.6]{PV}, there is an open equivariant embedding $X\hookrightarrow Y$ into an affine variety $Y$. For any $f\in\Ker(\partial)$, the orbits of the $\GG_a$-action on $Y$ corresponding to $f\partial$ coincide with the orbits of the original actions on $Y\setminus \{f=0\}$, while all points of the set $\{f=0\}$ become fixed. In particular, this action leaves the set $X$ invariant. This completes the proof of Proposition~\ref{lndga}.
\end{proof}

\begin{corollary} \label{corcc}
Let $X$ be an irreducible quasiaffine variety and $A=\KK[X]$. The variety $X$ admits a nontrivial $\GG_a$-action if and only if  there is a nonzero locally nilpotent derivation on $A$.
\end{corollary}

\section{Torus actions on rigid quasiaffine varieties} \label{s2}

In this section we generalize the results of \cite[Section~3]{FZ1} and \cite[Theorem~1]{AG} to the case of a quasiaffine variety. Let us recall that an irreducible algebraic variety $X$ is called \emph{rigid}, if it admits no nontrivial $\GG_a$-action.

\begin{theorem} \label{trigid}
Let $X$ be a rigid quasiaffine variety. There is a subtorus $\TT\subseteq\Aut(X)$ such that
for every torus action $T\times X\to X$ the image of $T$ in $\Aut(X)$ is contained in $\TT$. In other words, $\TT$ is a unique maximal torus in $\Aut(X)$.
\end{theorem}

Let us begin with some preliminary results.

\begin{lemma} \label{lemloc}
Let $X$ be an irreducible quasiaffine variety and $T\times X\to X$ be an action of a torus. Then there is a $T$-semi-invariant $f\in\KK[X]$ such that the localization $\KK[X]_f$ is finitely generated.
\end{lemma}

\begin{proof}
By \cite[Theorem~1.6]{PV}, there exists an open equivariant embedding $X\hookrightarrow Z$ into an irreducible affine $T$-variety $Z$. Let $I$ be the ideal of the subvariety $Z\setminus X$ in $\KK[Z]$. Since $I$ is $T$-invariant, there is a non-constant $T$-semi-invariant $f\in I$.
The principal open subset $Z_f$ is contained in $X$. Since the algebra $\KK[Z_f]$ is the localization $\KK[Z]_f$ and $\KK[X]$ is contained in $\KK[Z_f]$, we conclude that the algebra
$\KK[X]_f=\KK[Z]_f$ is finitely generated.
\end{proof}

Let $A=\oplus_{i\in\ZZ} A_i$ be a graded $\KK$-algebra and $\partial\colon A\to A$ a derivation. We define a linear map $\partial_k\colon A\to A$ by setting $\partial_k(a)$ to be
the homogeneous component $\partial(a)_{\deg(a)+k}$ of the element $\partial(a)$ for every homogeneous element $a\in A$. It is easy to check that $\partial_k$ is a derivation for all $k\in\ZZ$. We call it the $k$th homogeneous component of the derivation $\partial$.

\begin{proof}[Proof of Theorem~\ref{trigid}]
Assume that there are two torus actions $T_i\times X\to X$, $i=1,2$, such that the images of $T_i$ in $\Aut(X)$ are not contained in some torus $\TT$. The latter means that the actions do not commute. We may assume that $T_1$ and $T_2$ are one-dimensional. Let $A:=\KK[X]$ and
$$
A=\bigoplus_{u\in\ZZ} A_u \quad \text{and} \quad A=\bigoplus_{u\in\ZZ} A_u'
$$
be gradings corresponding to the actions of $T_1$ and $T_2$, respectively. Consider semisimple derivations $\partial$ and $\partial'$ on $A$ defined by $\partial(a)=ua$ for every $a\in A_u$ and
$\partial'(b)=ub$ for every $b\in A_u'$.

Let $\partial'_k$ be the $k$th homogeneous component of $\partial'$ with respect to the first grading. We claim that there are only finitely many nonzero homogeneous components and thus
the sum
$$
\partial'=\sum_{k\in\ZZ} \partial'_k
$$
has only finite number of nonzero terms.

Consider a localization $\KK[X]_f$ from Lemma~\ref{lemloc}, where $f$ is homogeneous with respect to the first grading. The algebra $\KK[X]_f$ is generated by some elements $f_1,\ldots,f_k\in \KK[X]$, which are homogeneous with respect to the first grading, and the element $\frac{1}{f}$.

Since $\KK[X]$ is contained in $\KK[X]_f$, every element $h\in\KK[X]$ is a linear combination of elements of the form
$$
\frac{f_1^{a_1}\ldots f_k^{a_k}}{f^a}
$$
and the image $\partial'(h)$ is a linear combination of the elements
$$
\sum_s\frac{a_s\partial'(f_s)f_1^{a_1}\ldots f_s^{a_s-1}\ldots f_k^{a_k}}{f^{a}}-\frac{a\partial'(f)f_1^{a_1}\ldots f_k^{a_k}}{f^{a+1}}.
$$
It shows that the shift of degree with respect to the first grading from $h$ to $\partial'(h)$ does not exceed the maximal shift of degree for $f_1,\ldots,f_k,f$. Hence the shift is bounded and we obtain the claim.

Let $\partial_m'$ be a nonzero homogeneous component of $\partial'$ with maximal absolute value
of the weight $m$. Since the derivations $\partial$ and $\partial'$ do not commute, we have $m\ne 0$. Then $(\partial_m')^r(a)$ is the highest (or the lowest) homogeneous component of the element $(\partial')^r(a)$ for every homogeneous $a\in A$. Since $a$ is contained in a finite dimensional $\partial'$-invariant subspace in $A$, the elements $(\partial')^r(a)$ cannot have nonzero projections to infinitely many components~$A_u$. Thus $(\partial_m')^r(a)=0$ for $r\gg 0$. We conclude that $\partial_m'$ is a nonzero locally nilpotent derivation of the algebra $A$. By Corollary~\ref{corcc}, we obtain a contradiction with the condition that $X$ is rigid.
\end{proof}

\begin{corollary}
In the setting of Theorem~\ref{trigid}, the maximal torus $\TT$ is a normal subgroup of $\Aut(X)$.
\end{corollary}

Let us finish this section with a description of affine algebraic groups which can be realized as automorphism groups of quasiaffine varieties. When this paper was already written, I found the same result in \cite[Theorem~1.3]{Kr}, cf. also \cite[Theorem~4.10~(a)]{LZ}.

\begin{proposition} \label{pdref}
Let $X$ be an irreducible quasiaffine variety. Assume that the automorphism group $\Aut(X)$ admits a structure of an affine algebraic group such that the action $\Aut(X)\times X\to X$ is a morphism of algebraic varieties. Then either $\Aut(X)$ is finite, or isomorphic to a finite extension of a torus, or isomorphic to the linear group
$$
G=\left\{
\left(
\begin{array}{cc}
1 & 0 \\
a & t
\end{array}
\right), \ \ a\in\KK, \ t\in\KK^{\times}
\right\}.
$$
\end{proposition}

\begin{proof} We assume first that $X$ is a rational curve. If $X=\AA^1$ then $\Aut(X)$ is isomorphic to the group $G$. If $X$ is $\AA^1$ with one point removed, then $\Aut(X)$ is an extension of 1-torus. If we remove more than one point from $\AA^1$, the group $\Aut(X)$ becomes finite. For a singular rational curve $X$, the automorphism group $\Aut(X)$ lifts to normalization and preserves the preimage of the singular locus. Thus $\Aut(X)$ is contained in an extension of  1-torus.

It follows from the description of the automorphism group of an elliptic curve and from Hurwitz's Theorem that the automorphism group of an affine curve $X$ of positive genus is finite.

Now let us assume that $\dim X\ge 2$. If $X$ is rigid then the affine algebraic group $\Aut(X)$
contains no one-parameter unipotent subgroup. It means that the unipotent radical and the semisimple part of $\Aut(X)$ are trivial. Hence $\Aut(X)$ is either finite or a finite extension of a torus.

Finally, let $\GG_a\times X\to X$ be a non-trivial action and $\partial\in\LND(\KK[X])$ the corresponding locally nilpotent derivation. By~\cite[Principle~11]{F}, the transcendence degree
of the algebra $\Ker(\partial)$ equals $\dim(X)-1\ge 1$. Let $U$ be a subspace in
$\Ker(\partial)$. Proposition~\ref{lndga},~(ii) implies that the automorphisms $\exp(f\partial)$, $f\in U$, form a commutative unipotent subgroup in $\Aut(X)$ of dimension $\dim(U)$. Since  $\dim(U)$ may be arbitrary, the group $\Aut(X)$ does not admit a structure of an affine algebraic group.
\end{proof}

\begin{remark}
Many examples of affine algebraic varieties whose automorphism group is a finite extension of a torus are provided by trinomial hypersurfaces, see~\cite[Theorem~3]{AG}.
\end{remark}

\begin{remark}
The class of affine algebraic groups which can be realized as the automorphism groups of complete
varieties is much wider. For example, the automorphism group of a complete toric variety is always an affine algebraic group of type A. A description of such groups is given in \cite{De,Cox}. Some other affine algebraic groups appear as the automorphism groups of Mori Dream Spaces; see e.g.
\cite[Theorem~7.2]{AHHL}. It is shown in~\cite[Theorem~1]{Br} that any connected algebraic group over a perfect field is the neutral component of the automorphism group scheme of some normal projective variety.
\end{remark}

\section{Main results} \label{s3}

We come to a characterization of transitivity properties for the automorphism group $\Aut(X)$ in terms of the special automorphism group $\SAut(X)$.

\begin{theorem} \label{tmain}
Let $X$ be an irreducible quasiaffine variety of dimension at least $2$. Assume that $X$ admits a nontrivial $\GG_a$- or $\GG_m$-action and the group $\Aut(X)$ acts on $X$ with an open orbit $\OO$. Then the following conditions are equivalent.
\begin{enumerate}
\item
The group $\Aut(X)$ acts 2-transitively on $\OO$.
\item
The group $\Aut(X)$ acts infinitely transitively on $\OO$.
\item
The group $\SAut(X)$ acts transitively on $\OO$.
\item
The group $\SAut(X)$ acts infinitely transitively on $\OO$.
\end{enumerate}
\end{theorem}

\begin{proof} Let us prove implications $(1)\Rightarrow (3)\Rightarrow (4) \Rightarrow (2) \Rightarrow (1)$. Implications ${(4)\Rightarrow (2)\Rightarrow (1)}$ are obvious. Implication $(3)\Rightarrow (4)$ is proved in \cite[Theorem~2.2]{AFKKZ} for $X$ affine and in \cite[Theorem~2]{APS}, \cite[Theorem~1.11]{FKZ} for $X$ quasiaffine.

It remains to prove $(1)\Rightarrow (3)$\footnote{This is the only implication where we use the condition on $\GG_a$- or $\GG_m$-action.}. Assume first that there is a nontrivial $\GG_a$-action on~$X$. Let us take two distinct points $x_1$ and $x_2$ in $\OO$ on one $\GG_a$-orbit. By assumption, for every distinct points $y_1,y_2\in\OO$ there exists an automorphism $\varphi\in\Aut(X)$ with
$\varphi(x_i)=y_i$, $i=1,2$. Then the points $y_1$ and $y_2$ lie in the same orbit for the $\GG_a$-action obtained from the initial one by conjugation with $\varphi$. It means that the group $\SAut(X)$ acts transitively on $\OO$.

Now assume that $X$ is rigid and admits a nontrivial $\GG_m$-action. If the maximal torus $\TT$ from Theorem~\ref{trigid} acts transitively on $\OO$, then $\OO$ is isomorphic to the torus $\TT$ and $\Aut(X)$ acts on $\OO$ transitively, but not 2-transitively. Indeed, let us fix an isomorphism between $\OO$ and $(\KK^{\times})^n$. The group $\Aut(\OO)$ is isomorphic to a semidirect product of $\TT$ and the group $\GL_n(\ZZ)$. It shows that the stabilizer in $\Aut(\OO)$ of the unit in $(\KK^{\times})^n$ preserves the set of points with rational coordinates. Consequently, the group $\Aut(\OO)$, and thus the group $\Aut(X)$, cannot act 2-transitively on $\OO$.

Now assume that the action of $\TT$ is not transitive on $\OO$. Let us take points
$x_1,x_2,x_3\in\OO$ such that $x_1\ne x_2$ lie in the same $\TT$-orbit and $x_3$ belongs to other $\TT$-orbit. By Corollary~\ref{corcc}, every automorphism of $X$ permutes $\TT$-orbits on $X$ and thus there is no automorphism preserving $x_1$ and sending $x_2$ to $x_3$, a contradiction with 2-transitivity.

This completes the proof of Theorem~\ref{tmain}.
\end{proof}

\begin{remark}
Implication $(1)\Rightarrow (3)$ for an affine variety $X$ admitting a nontrivial $\GG_a$-action was observed earlier in~\cite{BGT}.
\end{remark}

\begin{corollary} \label{ctrans}
Let $X$ be an irreducible quasiaffine variety of dimension at least $2$. Assume that $X$ admits a nontrivial $\GG_a$- or $\GG_m$-action. Then the following conditions are equivalent.
\begin{enumerate}
\item
The group $\Aut(X)$ acts 2-transitively on $X$.
\item
The group $\Aut(X)$ acts infinitely transitively on $X$.
\item
The group $\SAut(X)$ acts transitively on $X$.
\item
The group $\SAut(X)$ acts infinitely transitively on $X$.
\end{enumerate}
\end{corollary}

We recall that the \emph{Makar-Limanov invariant} $\text{ML}(A)$ of an algebra $A$ is the intersection of kernels of all locally nilpotent derivations on $A$. Using Proposition~\ref{lndga}, one can easily show that the Makar-Limanov invariant $\text{ML}(\KK[X])$ of the algebra of regular functions on an irreducible quasiaffine variety $X$ coincides with the algebra of invariants $\KK[X]^{\SAut(X)}$ of the special automorphism group. We denote $\text{ML}(\KK[X])$ just by $\text{ML}(X)$. Note that a~quasiaffine variety $X$ is rigid if and only if $\text{ML}(X)=\KK[X]$.

In \cite{Lie}, a field version of the Makar-Limanov invariant is introduced. Namely,
the \emph{field Makar-Limanov invariant} $\text{FML}(X)$ of an irreducible quasiaffine variety $X$ is the subfield of $\KK(X)$ consisting of all rational $\SAut(X)$-invariants. The condition $\text{FML}(X)=\KK$ implies $\text{ML}(X)=\KK$, but the converse is not true in general. By \cite[Corollary~1.14]{AFKKZ}, we have $\text{FML}(X)=\KK$ if and only if the group $\SAut(X)$ acts on $X$ with an open orbit. In this case the variety $X$ is unirational \cite[Proposition~5.1]{AFKKZ}. Together with Theorem~\ref{tmain} this yields the following result.

\begin{corollary} \label{cunirat}
Let $X$ be an irreducible quasiaffine variety. Assume that $X$ admits a nontrivial $\GG_a$- or $\GG_m$-action and the group $\Aut(X)$ acts on $X$ with an open orbit $\OO$. If the group $\Aut(X)$ is 2-transitive on $\OO$, then $X$ is unirational.
\end{corollary}

\begin{remark}
Corollary~\ref{cunirat} is a particular case of \cite[Theorem~5]{Po}. The latter theorem claims that if $X$ is an irreducible variety, the group $\Aut(X)$ acts generically 2-transitive on $X$, and $\Aut(X)$ contains a non-trivial connected algebraic subgroup, then $X$ is unirational. Moreover, if $X$ is irreducible, complete, and the group $\Aut(X)$ acts generically 2-transitive on $X$, then $X$ is unirational \cite[Corollary~3]{Po}.
\end{remark}

Let us finish this section with the following conjecture.

\begin{conjecture} \label{conj}
Conditions (1)-(4) of Theorem~\ref{tmain} are equivalent for any irreducible quasiaffine variety $X$ of dimension at least $2$.
\end{conjecture}

\begin{remark}
Jelonek~\cite{Je} has proved that every quasiaffine variety $X$ with an infinite automorphism
group is uniruled, i.e., for a generic point in $X$ there exists a rational curve in $X$ through this point.
\end{remark}

\section{Concluding remarks and questions} \label{s4}

In this section we discuss some results and questions related to Conjecture~\ref{conj}.
Let $\phi$ be an automorphism of a quasiaffine variety $X$ and $\phi^*$ be the induced
automorphism of the algebra $\KK[X]$. We say that $\phi$ is \emph{locally finite} if every element
of $\KK[X]$ is contained in a finite dimensional $\phi^*$-invariant subspace.

The following fact is well known to experts, but for the convenience of the reader we give it with a short proof.

\begin{proposition}
Let $X$ be an irreducible quasiaffine variety and $\phi$ an automorphism of~$X$. The following conditions are equivalent.
\begin{enumerate}
\item[(1)]
There exists a regular action $G\times X\to X$ of an affine algebraic group $G$ on $X$ such that $\phi$ is contained in the image of $G$ in the group $\Aut(X)$.
\item[(2)]
The automorphism $\phi$ is locally finite.
\end{enumerate}
\end{proposition}

\begin{proof}
For implication $(1)\Rightarrow (2)$, see e.g. \cite[Lemma~1.4]{PV}. Conversely, assume that
$\phi$ is locally finite and let $U$ be a finite-dimensional $\phi^*$-invariant subspace in $\KK[X]$ which generates a subalgebra $A$ in $\KK[X]$ such that the morphism $X\to Z:=\Spec(A)$ is an open embedding. Let $G$ be the subgroup of all automorphisms of $X$ that preserve the subspace~$U$. Since $U$ generates the field $\KK(X)$, the group $G$ is a subgroup of the general linear group~$\GL(U)$. Moreover, every element of $G$ induces an automorphism of $Z$. The subgroup $G'$ of all elements of $\GL(U)$ which induce an automorphism of $Z$ is closed in $\GL(U)$. The subgroup $G$ of $G'$ consists of automorphisms of $Z$ which preserve the (closed) subvariety $Z\setminus X$. This proves that $G$ is an affine algebraic group.
\end{proof}

\begin{remark}
For further characterizations of automorphisms belonging to algebraic subgroups of $\Aut(X)$, see~\cite{Ra}.
\end{remark}

Clearly, every automorphism of finite order is locally finite. The condition that a quasiaffine variety $X$ admits no nontrivial actions of the groups $\GG_a$ and $\GG_m$ means that every locally finite automorphism of $X$ has finite order.

\begin{problem} \label{p1}
Let $X$ be an irreducible quasiaffine variety such that every locally finite automorphism of $X$ has finite order. Can the group $\Aut(X)$ act transitively (2-transitively, infinitely transitively) on $X$?
\end{problem}

Let us give examples of automorphisms which are not locally finite. Let $X$ be a 2-torus with the algebra of regular functions
$\KK[X]=\KK[T_1,T_1^{-1},T_2,T_2^{-1}]$. Then the map
$$
\phi\colon (t_1,t_2) \mapsto (t_1t_2,t_2)
$$
is an automorphism of $X$ and the function $T_1$ is not contained in a finite dimensional $\phi^*$-invariant subspace of $\KK[X]$.

An automorphism of the affine plane $\AA^2$ which is not locally finite may be given as
$$
(x,y)\mapsto (x+y^2, x+y+y^2).
$$

More examples of automorphisms which are not locally finite can be found in~\cite{BD}. The authors describe a family of rational affine surfaces $S$ such that the normal subgroup $\Aut(S)_{\text{alg}}$ of $\Aut(S)$ generated by all algebraic subgroups of $\Aut(S)$ is not generated by any countable family of such subgroups, and the quotient $\Aut(S)/\Aut(S)_{\text{alg}}$ contains a free group over an uncountable set of generators. A description of automorphisms in~\cite{BD} is given in a purely geometric terms. It seems to be an important problem to find more methods for constructing automorphisms of quasiaffine varieties which are not locally finite.

Working with Conjecture~\ref{conj}, one may wish to replace an arbitrary quasiaffine variety by a quasiaffine variety admitting a nontrivial $\GG_a$- or $\GG_m$-action. For example, let $X$ be an irreducible quasiaffine variety such that the group $\Aut(X)$ is 2-transitive on $X$. Is it true that the group $\Aut(X\times\AA^1)$ is 2-transitive on $X\times\AA^1$? This question is related to algebraic families of automorphisms in the sense of~\cite{Ra}.

\smallskip

Let us finish this section with a general problem on transitivity for algebraic varieties. We say that an algebraic variety $X$ is \emph{homogeneous} if the group $\Aut(X)$ acts transitively on~$X$. A wide class of homogeneous varieties form homogeneous spaces of algebraic groups. At~the same time, not every homogeneous variety is homogeneous with respect to an algebraic group; an example of a homogeneous quasiaffine toric surface which is not a homogeneous space of an algebraic group is given in~\cite[Example~2.2]{AKZ}. More generally, it follows from ~\cite[Theorem~2.1]{AKZ} that every smooth quasiaffine toric variety is homogeneous. We plan to describe all homogeneous toric varieties in a forthcoming publication.

\begin{problem} \label{p2}
Describe all homogeneous algebraic varieties.
\end{problem}

Conjecture~\ref{conj} can be considered as a first step towards the solution of this problem.


\end{document}